\documentclass[12pt]{amsart}
\usepackage{hyperref}
\usepackage{amsmath}
\usepackage{amsthm}
\usepackage{amssymb}
\usepackage{tikz-cd}
\usepackage[capitalize,noabbrev]{cleveref}
\usepackage[backend=biber,maxbibnames=99,sorting=anyt]{biblatex}
\renewbibmacro{in:}{}
\bibliography{References}
\usepackage[inline]{enumitem}
\usepackage{color}

\newtheorem{theorem}{Theorem}[section]
\newtheorem*{theorem*}{Theorem}
\newtheorem{corollary}[theorem]{Corollary}
\newtheorem*{corollary*}{Corollary}

\newtheorem{proposition}[theorem]{Proposition}
\newtheorem*{proposition*}{Proposition}
\newtheorem{lemma}[theorem]{Lemma}

\theoremstyle{definition}
\newtheorem{definition}[theorem]{Definition} 
\newtheorem*{definition*}{Definition}
\newtheorem{example}[theorem]{Example}
\newtheorem*{example*}{Example}
\theoremstyle{remark}

\numberwithin{equation}{section}

\newcommand{\cat}{\mathbf}

\newcommand{\incl}{\hookrightarrow}

\newcommand{\To}{\Rightarrow}

\newcommand{\Top}{\cat{Top}}
\newcommand{\Cl}{\cat{Cl}}

\title{Relative cell complexes in closure spaces}
\author{Peter Bubenik}
\address{University of Florida, Department of Mathematics}
\email{peter.bubenik@ufl.edu}

\begin{document}
  
\begin{abstract}
    We give necessary and sufficient conditions for certain pushouts of topological spaces in the category of \v{C}ech's closure spaces to agree with their pushout in the category of topological spaces.
    We prove that in these two categories, the constructions of cell complexes by a finite sequence of closed cell attachments, which attach arbitrarily many cells at a time, agree. 
    Likewise, the constructions of CW complexes relative to a compactly generated weak Hausdorff space that attach only finitely many cells, also agree. 
    On the other hand, we give examples showing that the constructions of finite-dimensional CW complexes, CW complexes of finite type, and relative CW complexes that attach only finitely many cells, need not agree. 
\end{abstract}

\subjclass{Primary: 54A05; Secondary: 18A30, 54B17, 55N31}

\maketitle

\section{Introduction} \label{sec:introduction}

The classical theory of \v{C}ech's closure spaces~\cite{Cech:1966}, also called pretopological spaces, has found renewed interest as a setting for discrete homotopy theory and topological data analysis (TDA)~\cite{Rieser:2021,bubenik2023homotopy,rieser2021grothendieck,rieser2022cofibration,bubenik2023singular,ebel2023synthetic,PalaciosVela:2019,bubenik2021eilenbergsteenrod}. 
Examples of closure spaces include topological spaces, graphs, and directed graphs.
In TDA, one is interested in metric spaces and filtered topological spaces, which are both examples of filtered closure spaces~\cite{bubenik2023homotopy}.
To develop algebraic topology using closure spaces, it is important to understand how standard constructions, such as relative cell complexes, behave in this setting.

\subsubsection*{Outline of paper}

In \cref{sec:background}, we provide definitions, notation, and preliminary results.
In \cref{sec:do-not-agree}, we give examples that demonstrate that the constructions of 
finite relative CW complexes,
countable, finite-dimensional CW complexes,
and 
CW complexes of finite type
in $\Cl$ and $\Top$ do not agree.
In contrast, 
in \cref{sec:do-agree}
we prove that under suitable hypotheses, various constructions in $\Cl$ and $\Top$ do agree.

\subsubsection*{Related work}

Sterling Ebel and Chris Kapulkin~\cite{ebel2023synthetic} proved that the constructions of finite CW complexes in the category of pseudotopological spaces and in $\Top$ agree.
This implies that the constructions of finite CW complexes in $\Cl$ and $\Top$ agree, which is a special case of both \cref{cor:main,thm:main}.
They also give a model structure for pseudotopological spaces and show that it is Quillen equivalent to the standard model structure for $\Top$.

Antonio Rieser recently asked for examples in which the construction of a cell complex in $\Top$ and the category of pseudotopological spaces, which contains $\Cl$ as a full subcategory, do not agree~\cite[Remark 5.27]{rieser2022cofibration}. 
The pushouts in \cref{ex:2,ex:3} provide such examples.


\section{Background and preliminary work} \label{sec:background}

\subsection{Closure spaces and topological spaces}
\label{sec:cl-top}

A \emph{closure space} $(X,c)$ consists of a set $X$ and an operation $c$ on subsets of $X$ such that 
\begin{enumerate*}
    \item 
$c \varnothing = \varnothing$, 
\item 
for all $A \subset X$, $A \subset c A$, and 
\item \label{it:closure-def-3}
for all $A,B \subset X$, $c(A \cup B) = c A \cup c B$.
\end{enumerate*}
Note that from \eqref{it:closure-def-3}, it follows that for $A \subset B$, $c A \subset c B$.
If $c A = A$, we say that $A$ is \emph{closed}.

A (directed) graph is an example of a closure space. The set $X$ is the set of vertices and the closure operator of a set of vertices is given by the union of the vertices and (directed) neighborhoods of those vertices.
A \emph{topological space} is a closure space $(X,c)$ such that for all $A \subset X$, $c c A = c A$.
In cases where the closure operator is clear from the context, it will be omitted from the notation, e.g., if $X$ is a standard topological space such at the $k$-dimensional sphere $S^k$ or the $k$-dimensional disk $D^k$.

Given a closure space $(X,c)$ and $A \subset X$, let $\tau(c) A$ be the intersection of all closed sets containing $A$.
Then $\tau(c)$ is a closure operator, called the \emph{topological modification} of $c$.
It is the finest closure operator coarser than $c$ such that $(X,\tau(c))$ is a topological space.
A map of closure spaces $f:(X,c_X) \to (Y,c_Y)$ is a function $f:X \to Y$ such that for all $A \subset X$, $f c_X A \subset c_Y f A$.
Equivalently, for all $B \subset Y$, $c_X f^{-1} B \subset f^{-1} c_Y B$.
Using these definitions, we obtain the category $\Cl$ of closure spaces and their maps,
the full subcategory $\Top$ of topological spaces, 
the inclusion functor $i: \Top \to \Cl$, and 
its left adjoint $\tau: \Cl \to \Top$ given by $\tau(X,c) = (X,\tau(c))$.
For more details, see \cite[Section 16 B]{Cech:1966} or \cite[Section 1]{bubenik2023homotopy}.

Colimits of closure spaces are determined by the following two constructions~\cite[Definition 17.B.1, Theorem 33.A.4, Theorem 33.A.5]{Cech:1966}.
Given a collection of closure spaces $\{(X_i,c_i)\}_{i \in I}$, their coproduct consists of the disjoint union of sets $X = \coprod_i X_i$ together with the closure operator $c$ defined by $c(\coprod_i A_i) = \coprod_i c(A_i)$.
Given maps $f,g:(X,c_X) \to (Y,c_Y)$ of closure spaces, their coequalizer consists of the closure space $(Y/\!\!\sim,c_{\bar{Y}})$,
where $f(x) \sim g(x)$ for $x \in X$, and 
for $A \subset Y/\!\!\sim$, $c_{\bar{Y}}(A) = p(c_Y(p^{-1}(A)))$,
where 
$p:Y \to Y/\!\!\sim$
is the quotient map.
It follows~\cite[Theorem 3.4.12]{Riehl:2016} that $\Cl$ has all (small) colimits.
Note that the colimits in $\Cl$ are given by colimits in the category of sets together with appropriate closure operators. 
The category $\Cl$ also has all (small) limits~\cite[Theorem 32.A.4, Theorem 32.A.10]{Cech:1966}.

Consider a diagram in $\Cl$ whose objects lie in $\Top$.
Since $i$ is a right adjoint, the limits in $\Top$ and $\Cl$ agree~\cite[Theorem 4.5.2]{Riehl:2016}.
In contrast, the colimit in $\Top$ is given by the topological modification of the colimit in $\Cl$~\cite[Proposition 4.5.15(ii)]{Riehl:2016}.

It is easy to produce diagrams of topological spaces whose colimit in $\Cl$ does not agree with the colimit in $\Top$.
For example, let $A$ be the one-point topological space $\{1\}$, let $X$ and $Y$ be the two-point topological spaces $\{0,1\}$ and $\{1,2\}$ with the indiscrete topology, and let $i:A \to X$ and $f:A \to Y$ be the inclusions.
Consider the colimit of the diagram $X \stackrel{i}{\leftarrow} A \stackrel{f}{\rightarrow} Y$.
As will see in \cref{sec:pushout}, the colimit in $\Cl$ may be given by the closure space $(\{0,1,2\},c)$ where $c(0) = \{0,1\}$, $c(1) = \{0,1,2\}$ and $c(2) = \{1,2\}$, 
together with the inclusion maps.
In contrast, the colimit in $\Top$ may be given by 
$\{0,1,2\}$ with the indiscrete topology.
In \cref{sec:do-not-agree} we will
show that 
even various cell complex constructions
in $\Top$ and $\Cl$ need not agree.

\subsection{Closed inclusions}

Given closure spaces $(X,c_X)$ and $(Y,c_Y)$ say that a map $f:X \to Y$ is \emph{closed} if for all $A \subset X$ such that $c_XA = A$, $c_YfA = fA$.

Let $(X,c_X) \in \Cl$ and let $i:A \subset X$.
The \emph{subspace closure} on $A$, denoted $c_A$, is defined by $c_A B = (c_X B) \cap A$, for all $B \subset A$.
We obtain the inclusion $i:(A,c_A) \to (X,c_X)$.

\begin{lemma} \label{lem:closed-inclusion}
  The inclusion $i: A \incl X$ is a \emph{closed inclusion} if one of the following equivalent conditions hold.
  \begin{enumerate}
      \item \label{it:A} $i$ is closed, i.e. for all $B \subset A$ such that $c_A B = B$, $c_X B = B$.
      \item \label{it:B} $A$ is closed in $X$, i.e. $c_X A = A$.
      \item \label{it:C} $c_A = c_X|_A$, i.e. for all $B \subset A$, $c_A B = c_X B$.
  \end{enumerate}
\end{lemma}

\begin{proof}
    $\eqref{it:A} \To \eqref{it:B}$
    $c_A A = c_X A \cap A = A$. Therefore, by \eqref{it:A}, $c_X A = A$.

    $\eqref{it:B} \To \eqref{it:C}$ 
    Let $B \subset A$.
    Then, by \eqref{it:B}, $c_X B \subset c_X A = A$.
    Therefore, $c_A B = (c_X B) \cap A = c_X B$.

    $\eqref{it:C} \To \eqref{it:A}$ 
    Let $B \subset A$ such that $c_A B = B$.
    Then by \eqref{it:C}, $c_X B = c_A B = B$.
\end{proof}

\begin{corollary} \label{cor:closed-inclusion}
    The composition of closed inclusions is a closed inclusion.
\end{corollary}


\subsection{Pushouts} \label{sec:pushout}

From the definition of the coproduct and the coequalizer in $\Cl$, we obtain the following definition of the pushout of closure spaces.

\begin{definition} \label{def:pushout}
    Consider the maps of closure spaces $f$ and $i$ in the following diagram.
    \begin{equation} \label{cd:pushout}
        \begin{tikzcd} 
            (A,c_A) \ar[r,"f"] \ar[d,"i"] & (Y,c_Y) \ar[d,dashed,"j"] \\
            (X,c_X) \ar[r,dashed,"g"] & (Z,c_Z)   
        \end{tikzcd}
    \end{equation}
    The \emph{pushout} of $f$ and $i$
    may be given by the 
    pushout in the category of sets
        \begin{equation*}
        \begin{tikzcd} 
            A \ar[r,"f"] \ar[d,"i"] & Y \ar[d,dashed,"j"] \\
            X \ar[r,dashed,"g"] & Z   
        \end{tikzcd}
    \end{equation*}
    together with the closure operator $c_Z$ on $Z$    
    defined for $B \subset Z$ by
    \begin{equation} \label{eq:pushout-closure}
        c_Z B = j c_Y j^{-1}B \cup g c_X g^{-1} B.
    \end{equation}
\end{definition}

\subsection{Pushouts along closed inclusions}

For a pushout along a closed inclusion, we have the following result.



\begin{lemma} \label{lem:closed-inclusion-pushout}
    Consider the maps $f:A \to Y$ and $i:A \incl X$ in \eqref{cd:pushout}.
    Assume that $i$ 
    is a closed inclusion.
    Then for the pushout, we may take: $Z = Y \amalg (X \setminus A)$; $j: Y \incl Z$; $g:X \to Z$ given by $g(x) = f(x)$ if $x \in A$ and $g(x)=x$ if $x \in X \setminus A$; and for $B \subset Z$, 
    \begin{equation*}
        c_Z B = c_Y(B \cap Y) \cup g (c_X(B \cap (X \setminus A))).
    \end{equation*}
    Furthermore, $j$ is a closed inclusion.
\end{lemma}

\begin{proof}
    If $i:A \incl X$ is an inclusion, then we have the pushout in the category of sets
       \begin{equation*}
        \begin{tikzcd} 
            A \ar[r,"f"] \ar[d,"i"] & Y \ar[d,dashed,"j"] \\
            X \ar[r,dashed,"g"] & Z,   
        \end{tikzcd}
    \end{equation*}
    where $Z = Y \amalg (X \setminus A)$, $j$ is the inclusion, and $g$ is given by $g(x) = f(x)$ if $x \in A$ and $g(x)=x$ if $x \in S \setminus A$.

    Let $B \subset Z$. Then $B = (B \cap Y) \cup (B \cap (X \setminus A)$.
    Therefore $c_Z B = c_Z(B \cap Y) \cup c_Z(B \cap (X \setminus A))$.
    By \eqref{eq:pushout-closure}, $c_Z(B \cap (X \setminus A)) = g c_X (B \cap (X \setminus A))$.
    By \eqref{eq:pushout-closure}, $c_Z(B \cap Y) = c_Y(B \cap Y) \cup g c_X (g^{-1} (B \cap Y))$.
    For the right hand term, note that $g^{-1}(B \cap Y) = f^{-1}(B \cap Y) \subset A$.
    By \cref{lem:closed-inclusion}\eqref{it:C}, 
    $c_X(g^{-1}(B \cap Y)) = c_X( f^{-1}(B \cap Y)) = c_A(f^{-1}(B \cap Y)) \subset f^{-1}(c_Y(B \cap Y))$.
    Thus $g c_X g^{-1}(B \cap Y) \subset g f^{-1} c_Y(B \cap Y) = f f^{-1} c_Y(B \cap Y) \subset c_Y(B \cap Y)$.
    Therefore $c_Z(B \cap Y) = c_Y(B \cap Y)$.

    We have just shown that for $B \subset Y$, $c_Z(B) = c_Y(B)$.
    Therefore by \cref{lem:closed-inclusion}\eqref{it:C}, the inclusion $j:Y \incl Z$ is a closed inclusion.
\end{proof}

\subsection{Relative CW closure spaces}

Say that a closure space $(X,c)$ is obtained from a closure space $(A,b)$ by \emph{attaching cells} if there exists the following pushout,
\begin{equation} \label{eq:attaching}
    \begin{tikzcd} 
        \coprod_{k \in K} S_k^{n_k-1} \ar[r,"\varphi"] \ar[d,hookrightarrow,"i"] & (A,c_A) \ar[d,dashed,hookrightarrow,"j"] \\
        \coprod_{k \in K} D_k^{n_k} \ar[r,dashed,"\Phi"] & (X,c_X)  
    \end{tikzcd}
\end{equation}   
where the cardinality of $K$ is arbitrary, and for each $k \in K$, $D_k^{n_k}$ is the $n_k$-dimensional disk and the restriction of $i$ to
$S_k^{n_k-1}$ is the inclusion of its boundary.
Call $\varphi$ the \emph{cell-attaching map}.
Say that $(X,c_X)$ is obtained from $(A,c_A)$ by \emph{attaching $n$-cells}
if for each $k \in K$, $n_k=n$.

\begin{lemma} \label{lem:pushout-closed-inclusions}
    In the pushout \eqref{eq:attaching}, both $i$ and $j$ are closed inclusions.
\end{lemma}

\begin{proof}
    By \cref{lem:closed-inclusion}\eqref{it:B} and the definition of the coproduct in $\Cl$, $i$ is a closed inclusion.
    By \cref{lem:closed-inclusion-pushout}, $j$ is a closed inclusion.    
\end{proof}

Given a closure space $(X,c)$ and closed subspace $(A,c_A)$, the pair $((X,c),(A,c_A))$ is a \emph{relative CW closure space} if $(X,c)$ is the colimit of a diagram
\begin{equation} \label{eq:cw}
    (A,c_A) = (X^{-1},c_{-1}) \incl (X^0,c_0) \incl (X^1,c_1) \incl (X^2,c_2) \incl \cdots,
\end{equation}
where for each $n \geq 0$, $(X^n,c_n)$ is obtained from $(X^{n-1},c_{n-1})$ by attaching $n$-cells.
Note that at each step we attach arbitrarily many cells, but that the dimension of the cells is strictly increasing.
If at each step we attach only finitely many cells, then we say that $(X,A)$ is a relative CW closure space \emph{of finite type}.
If we attach only finitely many cells (in all of \eqref{eq:cw}), then we say that $(X,A)$ is a \emph{finite} relative CW closure space.
If we attach countably many cells, then we say that $(X,A)$ is a \emph{countable} relative CW closure space.
If for some $n$, $X = X^n$, then we say that $(X,A)$ is a \emph{finite-dimensional} relative CW complex.
In the case that $A = \varnothing$, say that $(X,A)$ is a \emph{CW closure space}.
Call the sequence $(\varphi^n)$ of cell-attaching maps a \emph{(relative) CW structure on $(X,A)$}.
By the definition of the colimit in $\Cl$, 
we may take $X = A \cup \bigcup_{n\geq 0} X^n$ 
and for $B \subset X$,
$cB = \bigcup_{n \geq -1} j_n c_n j_n^{-1} B$, where $j_n$ is the inclusion $X^n \subset X$.

\begin{proposition}
    For a relative CW structure \eqref{eq:cw} on $(X,A)$, for each $n \geq -1$, the canonical inclusion $j_n: X^n \incl X$ is a closed inclusion.
\end{proposition}

\begin{proof}
    Consider a closure space $(X,c)$ with closed subspace $(A,c_A)$ and a relative CW structure~\eqref{eq:cw}.
    By 
    \cref{lem:pushout-closed-inclusions}
    and \cref{cor:closed-inclusion}, 
    for each $-1 \leq i \leq j$, $X^i \incl X^j$ is a closed inclusion.
    Let $B \subset X^n$.
    Then $cB = \bigcup_{k \geq -1} j_k c_k j_k^{-1} B = \bigcup_{k \geq -1} j_k c_k (B \cap X^k)$.
    If $k \leq n$ then  
    by \cref{lem:closed-inclusion}\eqref{it:C},
    $c_k(B \cap X^k) = c_n(B \cap X^k) \subset c_n(B)$.
    If $k \geq n$ then $c_k(B \cap X^k) = c_k(B) = c_n(B)$
    by \cref{lem:closed-inclusion}\eqref{it:C}.
    Thus $cB = c_n B$.
    Therefore by \cref{lem:closed-inclusion}\eqref{it:C}, $j_n:X^n \incl X$ is a closed inclusion.
\end{proof}

\section{Constructions that may differ in \texorpdfstring{$\Top$}{Top} and \texorpdfstring{$\Cl$}{Cl}} \label{sec:do-not-agree}

We prove that various standard constructions on topological spaces can produce different spaces when performed in the category of closure spaces.
The definition of a pushout in $\Cl$ is given in \cref{def:pushout}. In the case of a pushout of a closed inclusion, which includes all of the examples in this section, an explicit description of this pushout is given in \cref{lem:closed-inclusion-pushout}.
As observed in \cref{sec:cl-top}, the pushout in $\Top$ may be obtained from the pushout in $\Cl$ by taking the topological modification~\cite[Proposition 4.5.15(ii)]{Riehl:2016}.

\begin{example} \label{ex:1}
    Consider the following diagram of topological spaces, 
    \[
    \begin{tikzcd}
        S^{n-1} \ar[r,"f"] \ar[d,hookrightarrow,"i"'] & \{0,1\} \\ 
        D^n 
    \end{tikzcd}
    \]
    in which 
    $\{0,1\}$ has the indiscrete topology, 
    $n \geq 1$,
    $i$ is the inclusion of the boundary of the $n$-dimensional disk, 
    and
    $f$ is the constant map with value $0$.
    Then the pushout in $\Cl$ is given by $(Z,c)$, where
    the set $Z$ may be taken to be the disjoint union of $\{0,1\}$ and the set $D^n \setminus S^{n-1}$.
    Furthermore, $c(D^n \setminus S^{n-1}) = (D^n \setminus S^{n-1}) \cup \{0\}$
    and $c^2(D^n \setminus S^{n-1}) = (D^n \setminus S^{n-1}) \cup \{0,1\}$.
    Therefore, the pushouts in $\Cl$ and $\Top$ do not agree.    
\end{example}

From \cref{ex:1}, we have the following.

\begin{proposition} \label{prop:ex-1}
    The constructions of finite relative CW complexes in $\Top$ and $\Cl$ need not agree.
\end{proposition}

\begin{example} \label{ex:2}
    Consider the following diagram of topological spaces, 
    \[
    \begin{tikzcd}
        \coprod_{k=1}^{\infty} S^1 \ar[r,"f"] \ar[d,hookrightarrow,"i"] & \left[0,1 \right] \\ 
        \coprod_{k=1}^{\infty} D^2 
    \end{tikzcd}
    \]
  where $f = \coprod_k f_k$ with $f_k$ given by the constant map $x \mapsto \frac{1}{k}$.
  Then in the pushout $(Z,c)$ in $\Cl$, 
  $c(\coprod_k (D^2 \setminus S^1)) = \coprod_k (D^2 \setminus S^1) \cup \{\frac{1}{k}\}_{k=1}^{\infty}$ and
  $c^2(\coprod_k (D^2 \setminus S^1)) = \coprod_k (D^2 \setminus S^1) \cup \{\frac{1}{k}\}_{k=1}^{\infty} \cup \{0\}$.
    Therefore, the pushouts in $\Cl$ and $\Top$ do not agree.
\end{example}

From \cref{ex:2}, we have the following.

\begin{proposition} \label{prop:ex-2}
    The constructions of countable finite-dimensional CW complexes in $\Top$ and $\Cl$ need not agree.
\end{proposition}

\begin{example} \label{ex:3}
    Consider the following diagram of topological spaces, 
    \[
    \begin{tikzcd}
        \coprod_{k=1}^{\infty} S^k \ar[r,"f"] \ar[d,hookrightarrow,"i"] & \left[0,1 \right] \\ 
        \coprod_{k=1}^{\infty} D^{k+1} 
    \end{tikzcd}
    \]
  where $f = \coprod_k f_k$ with $f_k$ given by the constant map $x \mapsto \frac{1}{k}$.
  Then in the pushout $(Z,c)$ in $\Cl$, 
  $c(\coprod_k (D^{k+1} \setminus S^k)) = \coprod_k (D^{k+1} \setminus S^k) \cup \{\frac{1}{k}\}_{k=1}^{\infty}$ and
  $c^2(\coprod_k (D^{k+1} \setminus S^k)) = \coprod_k (D^{k+1} \setminus S^k) \cup \{\frac{1}{k}\}_{k=1}^{\infty} \cup \{0\}$.
    Therefore, the pushouts in $\Cl$ and $\Top$ do not agree.
\end{example}

From \cref{ex:3}, we have the following.

\begin{proposition} \label{prop:ex-3}
    The constructions of CW complexes of finite type in $\Top$ and $\Cl$ need not agree.
\end{proposition}

Note the CW complexes in \cref{ex:2,ex:3} are both countable.
\cref{ex:2} is finite dimensional but is not of finite type;
\cref{ex:3} is of finite type but is not finite dimensional.
It follows that neither \cref{prop:ex-2} nor \cref{prop:ex-3} is a special case of the other.
The pushout of \cref{ex:3} also shows the following.

\begin{proposition}
    There exists a CW closure space $(X,c)$ with a CW structure 
    for which every $(X^n,c_n)$ is a topological space, 
    but $(X,c)$ is not a topological space.
\end{proposition}

\section{Constructions that agree in \texorpdfstring{$\Top$}{Top} and \texorpdfstring{$\Cl$}{Cl}} \label{sec:do-agree}

We prove that under suitable hypotheses, various standard constructions on topological spaces do  produce the same spaces when performed in the category of closure spaces.



\begin{theorem} \label{prop:1}
    Consider the pushout of closure spaces,
    \begin{equation*} 
        \begin{tikzcd} 
            (A,c_A) \ar[r,"f"] \ar[d,hookrightarrow,"i"] & (Y,c_Y) \ar[d,dashed,hookrightarrow,"j"] \\
            (X,c_X) \ar[r,dashed,"g"] & (Z,c_Z)   
        \end{tikzcd}
    \end{equation*}   
    where 
    $(A,c_A)$, $(X,c_X)$ and $(Y,c_Y)$ are topological spaces, and
    $i$ is a closed inclusion.
    \begin{enumerate}
        \item 
            If $f$ is closed then this pushout agrees with the one in $\Top$.
        \item \label{it:prop-1-b}
            If 
            this pushout agrees with the one in $\Top$ and 
            for every closed set $C \subset A$ there is a set $B \subset X \setminus A$ such that $(c_X B) \cap A = C$, 
            then $f$ is closed.
    \end{enumerate}

\end{theorem}

    Note that condition in \eqref{it:prop-1-b} is satisfied by a coproduct of inclusions of boundary spheres into disks. 

\begin{proof}[Proof of \cref{prop:1}]
    To prove the first statement, suppose that $f$ is closed.
    By \cref{lem:closed-inclusion-pushout}, we may
    take $Z = Y \amalg (X \setminus A)$ and corresponding maps $j$ and $g$.
    Let $B \subset Z$.
    We want to show that $c_Z c_Z B = c_Z B$.

    Let $B_1 = B \cap Y$ and $B_2 = B \cap (X \setminus A)$.
    By \cref{lem:closed-inclusion-pushout}, 
    \begin{equation} \label{eq:eB1}
        c_ZB = c_Y B_1 \cup g c_X B_2.
    \end{equation}
    Let $C_1 = (c_XB_2) \cap A$ and $C_2 = (c_XB_2) \cap (X \setminus A)$.
    Then
    \begin{equation} \label{eq:eB}
        c_Z B = c_Y B_1 \cup g C_2 \cup f C_1.
    \end{equation}
    Thus
    \begin{equation*}
        c_Z^2 B = c_Z c_Y B_1 \cup c_Z g C_2 \cup c_Z f C_1.
    \end{equation*}
    By the definition of closure, $c_Z B \subset c_Z^2 B$.
    It remains to show that $c_Z^2 B \subset c_Z B$.

    By \cref{lem:closed-inclusion-pushout}, 
    since $Y$ is a topological space, and by \eqref{eq:eB}, 
    $c_Z c_Y B_1 = c_Y^2 B_1 = c_Y B_1 \subset c_Z B$.
    By \cref{lem:closed-inclusion-pushout}, 
    by the definition $C_2$, 
    since $X$ is a topological space, 
    and by \eqref{eq:eB1},
    $c_Z g C_2 = g c_X C_2 \subset g c_X^2 B_2 = g c_X B_2 \subset c_Z B$.
    It remains to show that $c_Z f C_1 \subset c_Z B$.

    By definition, $C_1 \subset c_A C_1 \subset A$.
    Since $i$ is a closed inclusion, by the definition of $C_1$, and since $X$ is a topological space,
    $c_A C_1 = c_X C_1 \subset c_X^2 B_2 = c_X B_2$.
    Therefore, $c_A C_1 \subset c_X B_2 \cap A = C_1$.
    Thus $c_A C_1 = C_1$, i.e. $C_1$ is closed in $A$.
    Furthermore, since $f$ is closed, 
    $fC_1$ in $Y$.
    Finally,
    by \cref{lem:closed-inclusion-pushout}, 
    since $C_1$ is closed in $A$, 
    since $f$ is a map of closure spaces,
    since $(Y,c_Y)$ is a topological space,
    since $fC_1$ is closed in $Y$, 
    and by \eqref{eq:eB},
    $c_Z f C_1 = c_Y f C_1 = c_Y f c_A C_1 \subset c_Y^2 f C_1 = c_Y f C_1 = f C_1 \subset c_Z B$.

    For the second statement, 
    suppose that 
    the pushout agrees with the one in $\Top$ and that
    for every closed $C \subset A$ there is a $B \subset X \setminus A$ such that $(c_XB) \cap A = C$.
    Assume that $f$ is not closed.
    By \cref{lem:closed-inclusion-pushout}, we may
    take $Z = Y \amalg (X \setminus A)$ and corresponding maps $j$ and $g$.
    Since $f$ is not closed, there is a set $C \subset A$, with $c_AC = C$ and $c_YfC \supsetneq fC$.
    By assumption
    there is a set $B \subset X \setminus A$ with $(c_X B) \cap A = C$. We will show that $c_Z^2 B \neq c_Z B$, which is a contradiction. Therefore $f$ is closed.

    By \cref{lem:closed-inclusion-pushout}, $c_Z B = gc_XB$.
    Let $B' = c_XB \cap (X \setminus A)$.
    Then $c_XB = B' \cup C$ and hence $c_ZB = gB' \cup gC = gB' \cup fC$.
    Now $fC \subset Y \subset Z$ and since $B' \subset X \setminus A$, $gB' \subset X \setminus A \subset Z$. Hence
    $(c_Z B) \cap Y = fC$.
    By \cref{lem:closed-inclusion-pushout}, $c_Z^2 B \supset c_Z((c_Z B) \cap Y) = c_ZfC 
    =
    c_YfC$.
    Thus
    we have that $c_Z^2 B \cap Y \supset (c_YfC) \cap Y = c_YfC \supsetneq fC = (c_Z B) \cap Y$.
    Therefore $c_Z^2 B \neq c_Z B$.    
\end{proof}

\begin{theorem} \label{thm:main}
    Let $(X,c)$ be a closure space obtained from a 
    topological space by a finite sequence of  cell-attaching maps that are closed.
    Then $(X,c)$ is a topological space.
\end{theorem}

\begin{proof}
    This result
    follows from induction on 
\cref{prop:1}.
\end{proof}

Note that \cref{thm:main} allows arbitrarily many cells to be attached at a time.
For example, consider the real numbers obtained 
from the discrete space of integers
by the map that attaches countably many $1$-cells.

\begin{theorem} \label{prop:2}
    Consider the pushout of closure spaces,
    \begin{equation*} 
        \begin{tikzcd} 
            (A,c_A) \ar[r,"f"] \ar[d,hookrightarrow,"i"] & (Y,c_Y) \ar[d,dashed,hookrightarrow,"j"] \\
            (X,c_Z) \ar[r,dashed,"g"] & (Z,c_Z)   
        \end{tikzcd}
    \end{equation*}   
    where 
    $(A,c_A)$, $(X,c_Z)$ and $(Y,c_Y)$ are compactly generated weak Hausdorff topological spaces,
    $i$ is a closed inclusion, 
    and $f$ is closed.
    Then this pushout agrees with the one in $\Top$ and the one in the category of compactly generated weak Hausdorff spaces.    
\end{theorem}



\begin{proof}
    By \cref{prop:1}, this pushout agrees with the one in $\Top$.
    It is a result of Rezk~\cite[Proposition 10.6]{Rezk:2018}, that a pushout in $\Top$ 
    of a diagram of cofibrantly generated weak Hausdorff spaces 
    along a closed inclusion 
    agrees with the pushout in the full subcategory of compactly generated weak Hausdorff spaces.
    In particular, $(Z,c_Z)$ is a cofibrantly generated weak Hausdorff topological space.
\end{proof}

\begin{theorem} \label{cor:main}
    The constructions of finite CW complexes relative to a compactly generated weak Hausdorff topological space in $\Cl$ and $\Top$ agree.    
\end{theorem}

\begin{proof}
    Let $(A,c_A)$ be a compactly generated weak Hausdorff topological space.
    Consider a cell-attaching map $\varphi: S^{n-1} \to (A,c_A)$.
    Since $S^{n-1}$ is a compact topological space and $(A,c_A)$ is a weak Hausdorff topological space,
    by the definition of weak Hausdorff,
    $\varphi$ is closed.
    Now apply \cref{prop:2} to the cell attachment given by $\varphi$.
    \cref{cor:main} follows by induction on this argument.
\end{proof}
 
\section*{Acknowledgments}

    The author would like to thank the anonymous referee for helpful comments that led to many improvements.
 This research was partially supported 
 by the National Science Foundation (NSF) grant DMS-2324353
 and
 by the Southeast Center for Mathematics and Biology, an NSF-Simons Research Center for Mathematics of Complex Biological Systems, under NSF Grant No. DMS-1764406 and Simons Foundation Grant No. 594594.
The author would like to thank Nikola Mili\'{c}evi\'{c} for helpful discussions.

\printbibliography

\end{document}